\documentclass[11pt]{article}

\usepackage[a4paper,margin=1in]{geometry}
\usepackage{amsmath,amssymb,amsthm,mathtools}
\usepackage[T1]{fontenc}
\usepackage[utf8]{inputenc}
\usepackage{lmodern}
\usepackage{enumitem}
\usepackage{hyperref}

\newtheorem{theorem}{Theorem}[section]
\newtheorem{proposition}[theorem]{Proposition}
\newtheorem{lemma}[theorem]{Lemma}
\newtheorem{corollary}[theorem]{Corollary}
\newtheorem{remark}[theorem]{Remark}

\newtheorem{example}[theorem]{Example}

\newcommand{\OK}{\mathcal O_K}

\newcommand{\D}{\mathfrak D}

\title{Iterated extensions and the ramification dichotomy\thanks{This work was
supported by the project ``Group schemes, root systems, and related
representations'' funded by the European Union -- NextGenerationEU through
Romania's National Recovery and Resilience Plan (PNRR) call no.
PNRR-III-C9-2023-I8, Project CF159/31.07.2023, and coordinated by the Ministry
of Research, Innovation and Digitalization (MCID) of Romania.}}
\author{Mugurel Barcau and Vicen\c tiu Pa\c sol\\
\small Institute of Mathematics of the Romanian Academy, Bucharest\\
\small \texttt{mugurel.barcau@imar.ro; vicentiu.pasol@imar.ro}}
\date{}

\begin{document}
\maketitle

\begin{abstract}
Let $K/\mathbb Q_p$ be finite and let $f\in\OK[X]$ be monic, of degree at least
two, with $f'(X)\in\mathfrak m_K\OK[X]$, equivalently $\bar f\in k[X^p]$.
For a compatible inverse branch $f(t_{n+1})=t_n$ with $t_0\in\OK$, put
$K_n=K(t_n)$ and $K_\infty=\bigcup_nK_n$.  We prove that $K_\infty/K$ is either
unramified or deeply ramified.  More precisely, once ramification appears, the
ramification indices over the maximal unramified subfields tend to infinity and
the finite-level differents are unbounded.  In the Frobenius-type case
$f(X)\equiv X^{p^a}\pmod{\mathfrak m_K}$ the unramified alternative is trivial,
so $K_\infty=K$ or $K_\infty/K$ is deeply ramified.  After completion, the
non-unramified alternative gives perfectoid fields and examples show that
APF property need not hold at the algebraic level.
\end{abstract}

\section{Introduction and statement of the main theorem}

Let $K/\mathbb Q_p$ be finite, with ring of integers $\OK$, maximal ideal
$\mathfrak m_K$, and residue field $k$.  Fix an algebraic closure
$\overline K$, and let $K^{\rm ur}$ denote the algebraic maximal unramified
extension of $K$ inside $\overline K$.  We normalize the valuation $v_K$ by
\[
        v_K(\pi_K)=1.
\]

Let $f\in\OK[X]$
be monic of degree $d\geq2$, and assume
\begin{equation}\label{eq:derivative-hypothesis}
        f'(X)\in \mathfrak m_K\OK[X].              \tag{1.1}
\end{equation}
Equivalently, the reduction $\bar f\in k[X]$ satisfies $\bar f'(X)=0$, hence
$\bar f\in k[X^p]$.

Let $t_0\in\OK$ and choose a compatible backward orbit
\[
        f(t_{n+1})=t_n\qquad(n\geq0),
\]
with $t_n\in\overline K$.  Since $f$ is monic and $t_0$ is integral, every
$t_n$ is integral.  Put
\[
        K_n:=K(t_n),\qquad K_\infty:=\bigcup_{n\geq0}K_n.
\]
We call $K_\infty/K$ the iterated tower attached to $f$, $t_0$, and the chosen
inverse branch.  For every $n$, set
\[
        E_n:=K_n\cap K^{\rm ur},\qquad e_n:=[K_n:E_n].
\]
Then $E_n/K$ is unramified and $K_n/E_n$ is totally ramified.  Thus
$K_\infty/K$ is unramified if and only if $e_n=1$ for all $n$.

The main result is the following dichotomy for the iterated tower attached to the chosen inverse branch.

\begin{theorem}\label{thm:main}
With notation as above, $K_\infty/K$ is either unramified or deeply ramified.
\end{theorem}

Together with the Gabber--Ramero--Scholze characterization of complete rank-one
deeply ramified fields, the theorem has the following perfectoid reformulation.

\begin{corollary}\label{cor:perfectoid-completion}
Let $\widehat{K_\infty}$ be the completion of $K_\infty$ for the valuation
extending $v_K$.  Then either $K_\infty/K$ is unramified, or
$\widehat{K_\infty}$ is a perfectoid field.
\end{corollary}

\begin{proof}
If the tower is not unramified, Theorem~\ref{thm:main} gives that
$K_\infty/K$ is deeply ramified.  The completion has rank-one valuation, and
the characterization of Gabber--Ramero, in the form used in Scholze's
formulation of perfectoid fields, identifies complete rank-one deeply ramified
fields with perfectoid fields; see
\cite[Prop.~6.6.6]{GabberRameroAlmost} and \cite[Remark~3.3]{ScholzePerfectoid}.
\end{proof}

The proof first subtracts unramified residue lifts to reduce a ramified tail to
a positive-valuation inverse problem.  A uniform valuation estimate then forces
the ramification indices to be unbounded.  Finally, a different-growth argument
over the varying unramified subfields rules out bounded finite-level differents.

\subsection*{Relation with previous work}

The construction is related to the arithmetic-dynamical study of preimage
fields and arboreal representations, but our emphasis is local and branchwise:
we fix one $p$-adic inverse branch, do not pass to the Galois closure, and study
the purely inseparable reduction condition $\bar f'(X)=0$.  This is opposite in
spirit to the étale good-reduction hypotheses that lead to finite ramification;
see, for example, \cite{AitkenHajirMaire,BostonJones,BridyEtAl}.

The ramification language is closer to the field-of-norms and APF tradition.
The field of norms of Fontaine--Wintenberger, in the form developed by
Wintenberger~\cite{WintenbergerNorms}, is attached to arithmetically profinite
extensions.  Deep ramification was introduced by Coates and Greenberg
\cite{CoatesGreenberg} and later treated, with its relation to APF extensions,
by Fesenko~\cite{FesenkoDeeply}.  We use only the elementary consequence that,
for algebraic extensions of local fields, deep ramification is detected by the
unboundedness of finite-level differents.  The conclusion of
Theorem~\ref{thm:main} is deliberately stated in this weaker language.  Without
extra hypotheses the tower may contain an infinite unramified subextension, as
shown in Section~\ref{sec:apf-nonapf}; such a tower cannot be APF, even though
it is deeply ramified.

The dichotomy also has a differential, or ``étale'', interpretation.  In the
unramified alternative, $\mathcal O_{K_\infty}$ is a filtered union of finite
étale $\mathcal O_K$-algebras, hence
\[
        \Omega_{\mathcal O_{K_\infty}/\mathcal O_K}=0.
\]

In the ramified alternative, the relative module $\Omega_{\mathcal O_{K_\infty}/\mathcal O_K}$
is not expected to vanish. Indeed, for a deeply ramified algebraic extension
\(L/K\), Berger proves that
\[
\Omega_{\mathcal O_L/\mathcal O_K}\simeq L/\mathcal O_L
\]
as an \(\mathcal O_L\)-module \cite[Cor.~1.2]{BergerDifferentials}. Thus deep
ramification is not detected by the vanishing of differentials relative to the
original base field. Instead, if \(F=\widehat L\) denotes the completion of
\(L\), then Gabber--Ramero characterize deep ramification intrinsically by the
weakly étale property of $\mathcal O_F\to\mathcal O_{F^s}$, 
where \(F^s\) is a separable closure of \(F\); equivalently, by the vanishing
of the corresponding module of differentials
\cite[Prop.~6.6.2, Prop.~6.6.6]{GabberRameroAlmost}. In rank one, this is
precisely the condition that leads to perfectoid fields in Scholze's framework
\cite[Remark~3.3]{ScholzePerfectoid}. Thus Theorem~\ref{thm:main} may be viewed
as an étaleness dichotomy: an inverse branch is either ind-étale over the
original valuation ring, or the completion of the resulting extension
enters the weakly étale/perfectoid world of Gabber--Ramero.
This differential viewpoint is also central in the work of
Iovita--Zaharescu on the Galois theory of $B_{\rm dR}^{+}$ (see \cite{IovitaZaharescuBdR}).

The non-APF
phenomenon itself is already present in Fesenko's work, where deeply ramified
extensions which are not APF are constructed by ramification-theoretic and
class-field-theoretic methods, including examples involving infinite residue
extensions or nondiscrete upper breaks~\cite[Secs.~2.1--2.4]{FesenkoDeeply}.
The examples below have a different role: they show that the same obstruction
can occur inside a single iterated-polynomial inverse branch.  In such a tower,
prime-to-$p$ residue dynamics may produce an infinite unramified part, while the
completion remains governed by deep ramification.

The closest predecessors are the Frobenius-iterate extensions of Cais--Davis
\cite{CaisDavis}, the strictly APF criterion of Cais--Davis--Lubin
\cite{CaisDavisLubin}, and Berger's iterated extensions attached to relative
Lubin--Tate groups~\cite{BergerIterated}.  Those settings impose much stronger
Frobenius, Galois, or formal-group structure, often starting from a uniformizer
and hence carrying extra rigidity.  Here $t_0$ is arbitrary and the
residue dynamics of $\bar f\in k[X^p]$ may include prime-to-$p$ phenomena.  APF
therefore cannot be expected; the theorem isolates the robust conclusion that
survives, namely a perfectoid-or-unramified alternative after completion.

A special case of particular interest is the \emph{Frobenius-type} situation
\[
        f(X)\equiv X^{p^a}\pmod{\mathfrak m_K}.
\]
This includes the kind of iterated towers considered by Berger
\cite{BergerIterated}.  In this case the unramified alternative collapses:
any unramified compatible tower is already contained in the base field.

\section{Basic lemmas}

We use without comment that, since $f$ is monic and $t_0\in\OK$, all $t_n$ are
integral; the same holds after integral translations.

\subsection{Derivative growth}

The derivative hypothesis is inherited by iterates and translated compositions.

\begin{lemma}\label{lem:derivative-growth}
For every $n\geq1$ and every $r\geq1$,
\[
        (f^{\circ n})^{(r)}(X)\in \pi_K^n\OK[X].
\]
More generally, if $\alpha_1,\ldots,\alpha_n$ are integral over an unramified extension of $K$ and
\[
        F_i(X):=f(\alpha_i+X)-f(\alpha_i),
\]
then for
\[
        H_n:=F_1\circ F_2\circ\cdots\circ F_n
\]
one has
\[
        H_n^{(r)}(X)\in \pi_K^n O[X]\qquad(r\geq1),
\]
where $O$ is the ring of integers of any field containing the $\alpha_i$.
\end{lemma}

\begin{proof}
For $r=1$, the chain rule gives a product of $n$ values of $f'$ at integral
polynomials, hence a factor $\pi_K^n$.  Higher derivatives are obtained by
formal differentiation, which cannot decrease $p$-adic valuation.
\end{proof}

\subsection{A valuation estimate for inverse branches}

The following estimate is uniform under integral translations of $f$.

\begin{lemma}\label{lem:contraction}
Let $L/K$ be an algebraic extension with valuation still denoted $v_K$ and normalized by $v_K(\pi_K)=1$.  Let
\[
        F(X)=\sum_{i=1}^d c_iX^i\in\mathcal O_L[X]
\]
satisfy $F(0)=0$ and $F'(X)\in\pi_K\mathcal O_L[X]$.  Let $x\in\mathfrak m_L$ and put $y=F(x)\neq0$.  Then
\[
        v_K(x)\leq \psi\bigl(v_K(y)\bigr),
        \qquad
        \psi(s):=\max\{s-1,s/2\}.
\]
Consequently, if
\[
        y_j=F_j(y_{j+1})\qquad(j=0,\ldots,r-1)
\]
inside a common algebraic extension of $K$, with each $F_j$ satisfying the same
hypotheses, with all $y_j$ in its maximal ideal, and with $y_0\neq0$, then
\[
        v_K(y_r)\leq \psi^{\circ r}\bigl(v_K(y_0)\bigr).
\]

\end{lemma}

\begin{proof}
Write $s=v_K(x)>0$ and $T=v_K(y)$.  By the ultrametric inequality,
\[
        T=v_K(y)\geq
        \min_{c_i\neq0} v_K(c_ix^i).
\]
Thus for some $i$ with $c_i\neq0$,
\[
        v_K(c_i)+is\leq T.
\]
If $i=1$, then $c_1$ is divisible by $\pi_K$, because $F'(X)\in\pi_K\mathcal O_L[X]$; hence
\[
        1+s\leq T,
        \qquad\text{so}\qquad
        s\leq T-1.
\]
If $i\geq2$, then $v_K(c_i)\geq0$, and therefore
\[
        2s\leq is\leq T,
        \qquad\text{so}\qquad
        s\leq T/2.
\]
Combining the two cases gives
\[
        s\leq\max\{T-1,T/2\}=\psi(T).
\]
The iterated estimate follows step by step; nonvanishing propagates backwards
from $y_0\neq0$.  
\end{proof}

\begin{remark}\label{rem:psi}
For the function
\[
\psi(s)=\max\{s-1,s/2\},
\]
one has $\psi(s)<s$ for every $s>0$. Furthermore,
\[
\psi^{\circ r}(C)\longrightarrow0
\qquad(r\to\infty)
\]
for every $C>0$. Indeed, while $s>2$, each iteration decreases the value by at least $1$, and once $s\le2$, each iteration halves it.
\end{remark}

\subsection{Separation from the unramified part}

The following lemma is probably well known. Since we could not find a suitable reference, we include a short proof. It shows that an element generating a nontrivial totally ramified extension cannot be approximated arbitrarily well by elements of unramified extensions.

\begin{lemma}\label{lem:separation}
Let $E/K$ be a finite unramified extension and let $L/E$ be a finite totally ramified extension.  Let $a\in L\setminus E$. Then there is a constant $C(a)>0$ such that for every
finite unramified extension $E'/E$ in $K^{\rm ur}$ and every $b\in E'$, one has
\[
        v_K(a-b)\leq C(a).
\]
\end{lemma}

\begin{proof}
Put $F=E(a)$; then $F/E$ is nontrivially totally ramified and
$F\cap K^{\rm ur}=E$.  Choose $C(a)>0$ larger than all numbers
$v_K(a-\sigma(a))$, where $\sigma$ runs through the $E$-embeddings of $F$ in
$\overline K$ with $\sigma(a)\neq a$.  If, for some $b\in E'\subset K^{\rm ur}$, one had
\[
v_K(a-b)>C(a),
\]
then Krasner's lemma would imply
\[
E(a)\subset E(b)\subset E',
\]
a contradiction, since $E(a)/E$ is nontrivially totally ramified whereas
$E'/E$ is unramified. Hence $v_K(a-b)\leq C(a)$ for all such $b$.
\end{proof}

\subsection{Translated tails}
We now return to the tower
\[
        K_n=K(t_n),\qquad
        E_n=K_n\cap K^{\rm ur},\qquad
        e_n=[K_n:E_n].
\]
The next lemma provides the key reduction after ramification first appears. If
$e_m>1$, then, for every later level $N>m$, one can subtract suitable residue
lifts from the unramified field $E_N$ so that the tail
\[
        t_m,t_{m+1},\ldots,t_N
\]
is transformed into a positive-valuation inverse problem to which
Lemma~\ref{lem:contraction} applies.
\begin{lemma}\label{lem:translated-tail}
Assume that $e_m>1$.  For every $N>m$, put $r=N-m$.  Then there are elements
\[
        \alpha_0,\ldots,\alpha_r\in \mathcal O_{E_N}
        \qquad\text{and}\qquad
        u_0,\ldots,u_r\in K_N
\]
with the following properties:
\[
        u_j=t_{m+j}-\alpha_j\in\mathfrak m_{K_N}
        \qquad(0\leq j\leq r),
\]
\[
        u_j=F_j(u_{j+1}),\qquad
        F_j(X):=f(X + \alpha_{j+1})-f(\alpha_{j+1})
        \qquad(0\leq j<r),
\]
and
\[
        F_j(0)=0,\qquad F_j'(X)\in \pi_K\mathcal O_{E_N}[X].
\]
Moreover, there is a constant $C_m>0$, independent of $N$, such that
\[
        0<v_K(u_0)\leq C_m,
\]
and one also has $u_r\neq0$.
\end{lemma}

\begin{proof}
Since $K_N/E_N$ is totally ramified, $K_N$ and $E_N$ have the same residue
field.  Choose $\alpha_r\in \mathcal O_{E_N}$ lifting the residue class of $t_N$, and
define
\[
        \alpha_j:=f(\alpha_{j+1})\qquad(j=r-1,\ldots,0).
\]
Then $\alpha_j$ and $t_{m+j}$ have the same residue, and the displayed
relations follow immediately from the definition of
$u_j:=t_{m+j}-\alpha_j$.  The derivative condition follows from
\eqref{eq:derivative-hypothesis}.

It remains to bound $u_0$. Since $e_m>1$, the extension $K_m/E_m$ is nontrivially totally ramified, $t_m\notin E_m$, the extension $E_N/E_m$ is unramified, and $\alpha_0\in E_N$, Lemma~\ref{lem:separation} gives
\[
        v_K(u_0)=v_K(t_m-\alpha_0)\le C_m
\]
for all $N$. Since $u_0\in\mathfrak m_{K_N}$, we obtain
\[
        0<v_K(u_0)\le C_m.
\]

Finally, if $u_r=0$, then the relations
$u_j=F_j(u_{j+1})$ and $F_j(0)=0$ force $u_0=0$, a contradiction.
\end{proof}

\section{First ramification forces unbounded ramification}

The first step toward the proof of Theorem~\ref{thm:main} is to show that, once
ramification appears, the ramification indices cannot remain bounded. The
argument combines the translated-tail construction of
Lemma~\ref{lem:translated-tail} with the valuation estimate of
Lemma~\ref{lem:contraction}.

\begin{proposition}\label{prop:unbounded-e}
If the tower is not unramified, then the ramification indices
\[
e_n=[K_n:E_n]
\]
are unbounded.
\end{proposition}

\begin{proof}
Assume, for contradiction, that the sequence $(e_n)$ is bounded.  Since
$K_n\subset K_{n+1}$, ramification indices multiply in the tower; hence
$(e_n)$ is nondecreasing.  Thus there are an index $m$ and an integer $e$ such
that
\[
        e_N=e_m=e\qquad(N\geq m).
\]
As the tower is not unramified, we choose $m$ after the first ramified level, so
$e>1$.  Then $K_N/K_m$ has ramification index $1$ for every $N\geq m$, and
therefore $K_N/K_m$ is unramified.

Fix $N>m$ and write $r=N-m$.  Apply Lemma~\ref{lem:translated-tail} at the
level $m$.  It gives a translated tail
\[
        u_j=F_j(u_{j+1})\qquad(0\leq j<r)
\]
with all $u_j$ in the maximal ideal, with $u_r\neq0$, and with
\[
        0<v_K(u_0)\leq C_m
\]
for a constant $C_m$ independent of $N$.
Applying Lemma~\ref{lem:contraction} to the chain above gives
\[
        v_K(u_r)\leq \psi^{\circ r}\bigl(v_K(u_0)\bigr)
        \leq \psi^{\circ r}(C_m).
\]
But $u_r\in K_N$ has positive valuation, and $e_N=e$.  Hence
\[
        \frac1e\leq v_K(u_r)\leq \psi^{\circ r}(C_m).
\]
The right hand side tends to zero as $r\to\infty$, a contradiction.
\end{proof}

\section{Different growth and proof of the main theorem}

We now turn from ramification indices to differents.  We shall use the standard
criterion for deep ramification \cite{CoatesGreenberg,FesenkoDeeply}: if
$L/K$ is algebraic and
\[
        K=L_0\subset L_1\subset L_2\subset\cdots\subset L
\]
is a cofinal tower of finite subextensions, then $L/K$ is deeply ramified if
and only if the normalized valuations $v_K(\D_{L_i/K})$ are unbounded.  Thus,
for our tower, it remains in the non-unramified case to prove that the
valuations $v_K(\D_{K_N/K})$ are unbounded.

Assume, until the proof of Theorem~\ref{thm:main}, that the tower is not
unramified.  By Proposition~\ref{prop:unbounded-e}, $e_N\to\infty$.  Let
$n_0$ be the first index for which $e_{n_0}>1$, and put
\[
        E_0:=E_{n_0},\qquad e_*:=e_{n_0}.
\]

For $N>n_0$, write $r=N-n_0$.  Apply Lemma~\ref{lem:translated-tail} at the
first ramified level $n_0$ and keep its notation.  Thus
$u_j\in\mathfrak m_{K_N}$,
\begin{equation}\label{eq:translated-chain}
        u_j=F_j(u_{j+1}),\qquad
        F_j(X):=f(X + \alpha_{j+1})-f(\alpha_{j+1}),
\end{equation}
where $F_j(0)=0$ and
\[
        F_j'(X)=f'(X+\alpha_{j+1})\in\pi_K\mathcal O_{E_N}[X].
\]
Moreover, with $C_0:=C_{n_0}$, we have
\begin{equation}\label{eq:u0-bound}
        0<v_K(u_0)\leq C_0,
\end{equation}
and $u_r\neq0$. We also put
\[
        B_N:=K_{n_0}E_N.
\]
Then $B_N/K$ has ramification index $e_*$, independent of $N$, and $K_N/B_N$
is totally ramified of degree
\[
        m_N:=e_N/e_*\longrightarrow\infty.
\]

We shall use the following elementary property of the Gauss valuation. It
allows one to recover divisibility of the coefficients of a polynomial from the
valuation of its value at a uniformizer of a totally ramified extension.

For a nonarchimedean local field \(B\) with normalized valuation \(v_B\),
recall that the Gauss valuation on \(\mathcal O_B[X]\) is defined by
\[
        v_{B,G}\!\left(\sum_i b_iX^i\right):=\min_i v_B(b_i),
\]
with the convention that \(v_{B,G}(0)=+\infty\).

\begin{lemma}\label{lem:gauss-general}
Let $B$ be a nonarchimedean local field, let $M/B$ be totally ramified of degree
$m$, and let $\pi_M$ be a uniformizer of $M$. If $h\in \mathcal O_B[X]$ has degree $<m$
and
\[
        v_B(h(\pi_M))\ge \lambda
\]
for an integer $\lambda$, then
\[
        v_{B,G}(h)\ge \lambda.
\]
\end{lemma}

\begin{proof}
Write $h=\sum_{i=0}^{m-1}b_iX^i$.  The terms $b_i\pi_M^i$ have distinct
valuations modulo the value group of $B$, so no cancellation can occur among
terms of minimal valuation.  Hence
\[
        v_B(h(\pi_M))
        =\min_{b_i\neq0}\left(v_B(b_i)+\frac{i}{m}\right).
\]
If this minimum is at least the integer $\lambda$, then
$v_B(b_i)+i/m\geq\lambda$ for each nonzero $b_i$.  Since
$v_B(b_i)\in\mathbb Z$ and $0\leq i/m<1$, each $v_B(b_i)\geq\lambda$.
\end{proof}

Choose a uniformizer $\rho_N$ of $K_N/B_N$, and let
\[
        \Phi_N(X)\in \mathcal O_{B_N}[X]
\]
be its Eisenstein minimal polynomial over $B_N$.  Since
$\mathcal O_{K_N}=\mathcal O_{B_N}[\rho_N]$, there is a polynomial
\[
        P_N(X)\in \mathcal O_{B_N}[X],\qquad \deg P_N<m_N,
\]
such that
\[
        P_N(\rho_N)=u_r.
\]
Because $u_r$ has positive valuation and $\rho_N$ has zero residue,
\begin{equation}\label{eq:PN-zero}
        P_N(0)\in\mathfrak m_{B_N}.
\end{equation}

Let
\[
        H_N:=F_0\circ F_1\circ\cdots\circ F_{r-1}.
\]
Then $H_N(u_r)=u_0$.  Put
\[
        Q_N(X):=H_N(P_N(X))-u_0.
\]
Repeated Euclidean division by the monic polynomial $\Phi_N$ gives a finite
expansion
\[
        Q_N(X)=\sum_{j\geq0}\Phi_N(X)^j g_{N,j}(X),
        \qquad g_{N,j}\in \mathcal O_{B_N}[X],\quad \deg g_{N,j}<m_N .
\]
Evaluating at $\rho_N$ gives $g_{N,0}(\rho_N)=0$.  Since
$\deg g_{N,0}<m_N=[K_N:B_N]$, we have $g_{N,0}=0$.  Thus
\begin{equation}\label{eq:phi-expansion}
        Q_N(X)=\sum_{j\geq1}\Phi_N(X)^j g_{N,j}(X).
\end{equation}

Bounded differents force the coefficients in \eqref{eq:phi-expansion} to become
highly divisible.

\begin{lemma}\label{lem:growth-g}
Assume that the valuations $v_K(\D_{K_N/K})$ are bounded.  Then, for every fixed
$j\geq1$, there exists a constant $R_j$ such that
\[
        v_{B_N,G}(g_{N,j})\geq r-R_j
\]
for all $N>n_0$.
\end{lemma}

\begin{proof}

Let \(v_N\) denote the valuation on \(K_N\) extending the normalized valuation
\(v_{B_N}\) of \(B_N\). Since $B_N=K_{n_0}E_N$
and \(E_N/E_0\) is unramified, the extension \(B_N/E_N\) is obtained from the
fixed totally ramified extension \(K_{n_0}/E_0\) by unramified base change. In
particular, its ramification index is \(e_*\), independently of \(N\). Since
\(E_N/K\) is unramified, \(v_{E_N}=v_K\), and hence
\[
        v_N=e_*\,v_K.
\]

Moreover, the different \(\mathfrak D_{B_N/E_N}\) is obtained by unramified
base change from the fixed different \(\mathfrak D_{K_{n_0}/E_0}\), while
\(\mathfrak D_{E_N/K}=\mathcal O_{E_N}\). Therefore
\(v_N(\mathfrak D_{B_N/K})\) is bounded independently of \(N\). By transitivity
of differents,
\[
        \mathfrak D_{K_N/K}
        =
        \mathfrak D_{K_N/B_N}\,
        \mathfrak D_{B_N/K}\mathcal O_{K_N}.
\]
The assumed boundedness of \(v_K(\mathfrak D_{K_N/K})\), together with
\(v_N=e_*v_K\) on \(K_N\), therefore implies that
\(v_N(\mathfrak D_{K_N/B_N})\) is bounded independently of \(N\).

As \(\mathcal O_{K_N}=\mathcal O_{B_N}[\rho_N]\), we have
\[
        \mathfrak D_{K_N/B_N}=(\Phi_N'(\rho_N)).
\]
Thus there is a constant \(D\) such that
\begin{equation}\label{eq:phi-prime-bound}
        v_N(\Phi_N'(\rho_N))\leq D
\end{equation}
for all \(N>n_0\).

For every fixed $q\geq1$,
\begin{equation}\label{eq:Q-derivative-bound}
        v_N\bigl(Q_N^{(q)}(\rho_N)\bigr)\geq r
\end{equation}
for all $N$.  Indeed, Lemma~\ref{lem:derivative-growth} gives
$H_N^{(s)}\in\pi_K^r\mathcal O_{E_N}[X]$ for every $s\geq1$, while the derivatives of
$P_N$ have coefficients in $\mathcal O_{B_N}$ and $\rho_N$ is integral.  Hence every
term in the chain-rule expansion of $Q_N^{(q)}(\rho_N)$ is divisible by
$\pi_K^r$, and $v_N(\pi_K^r)=e_*r\geq r$.

We prove the claim by induction on $j$.  After differentiating
\eqref{eq:phi-expansion} $j$ times and evaluating at $\rho_N$, all terms with
index $>j$ vanish and the $j$-th term contributes
\[
        j!\,\Phi_N'(\rho_N)^j g_{N,j}(\rho_N).
\]
The remaining terms only involve $g_{N,\ell}$ with $\ell<j$ and their
derivatives, evaluated at $\rho_N$.  For $j=1$ there are no such terms.  For
$j>1$, the induction hypothesis implies that these remaining terms all have
valuation at least $r$ minus a constant depending only on $j$: derivatives do
not decrease Gauss valuation, and all derivatives of $\Phi_N$ evaluated at
$\rho_N$ are integral.  Combining this with
\eqref{eq:Q-derivative-bound}, we get
\[
        v_N\bigl(j!\,\Phi_N'(\rho_N)^j g_{N,j}(\rho_N)\bigr)\geq r-R'_j
\]
for some constant $R'_j$.  Using \eqref{eq:phi-prime-bound}, we deduce
\[
        v_N(g_{N,j}(\rho_N))\geq r-e_* v_K(j!)-jD-R'_j.
\]
Applying Lemma~\ref{lem:gauss-general} to the integer part of the lower bound, and enlarging the constant if necessary,
we obtain
\[
        v_{B_N,G}(g_{N,j})\geq r-R_j .
\]
\end{proof}

We now use the coefficient growth obtained in Lemma~\ref{lem:growth-g} to rule
out bounded different valuations.

\begin{proposition}\label{prop:unbounded-different}
If the tower is not unramified, then the valuations $v_K(\D_{K_N/K})$ of the differents are unbounded.
\end{proposition}

\begin{proof}
Assume, for contradiction, that the valuations $v_K(\D_{K_N/K})$ are bounded. Evaluate \eqref{eq:phi-expansion} at $X=0$:
\[
        Q_N(0)=H_N(P_N(0))-u_0
        =\sum_{j\geq1}\Phi_N(0)^j g_{N,j}(0).
\]

Since $\Phi_N$ is Eisenstein over $B_N$, $\Phi_N(0)$ is a uniformizer of $B_N$.
Fix an integer $\lambda>0$.  For $j>\lambda$, the $j$-th summand on the right
has $v_N$-valuation $>\lambda$.  For $1\leq j\leq\lambda$,
Lemma~\ref{lem:growth-g} makes the valuations of $g_{N,j}(0)$ tend to infinity
with $N$.  Hence $v_N(Q_N(0))\to\infty$, and therefore, because
$e(B_N/K)=e_*$ is fixed,
\[
        v_K\bigl(H_N(P_N(0))-u_0\bigr)\longrightarrow\infty.
\]

Choose $N>n_0$ sufficiently large so that
\[
        v_K\bigl(H_N(P_N(0))-u_0\bigr)>C_0
\]
and
\[
        \psi^{\circ r}(C_0)<\frac{1}{e_*}.
\]
By \eqref{eq:u0-bound}, the first inequality implies
\[
        v_K\bigl(H_N(P_N(0))\bigr)=v_K(u_0).
\]
If $P_N(0)=0$, then $H_N(P_N(0))=H_N(0)=0$, because each $F_j(0)=0$, and hence
$H_N(P_N(0))-u_0=-u_0$, contradicting the choice of $N$.  Thus
$P_N(0)\neq0$.

By \eqref{eq:PN-zero}, $P_N(0)\in\mathfrak m_{B_N}$.  Since $e(B_N/K)=e_*$,
every nonzero element of $\mathfrak m_{B_N}$ has $v_K$-valuation at least
$1/e_*$.  Hence
\begin{equation}\label{eq:PN-lower}
        v_K(P_N(0))\geq \frac{1}{e_*}.
\end{equation}

On the other hand, set $z_r=P_N(0)$ and $z_j=F_j(z_{j+1})$ for
$j=r-1,\ldots,0$.  Then all $z_j$ lie in the maximal ideal, and
$z_0=H_N(P_N(0))$ is nonzero.  Lemma~\ref{lem:contraction} gives
\[
        v_K(P_N(0))
        \leq \psi^{\circ r}\bigl(v_K(z_0)\bigr)
        = \psi^{\circ r}\bigl(v_K(u_0)\bigr)
        \leq \psi^{\circ r}(C_0)
        < \frac{1}{e_*}.
\]
This contradicts \eqref{eq:PN-lower}. 
\end{proof}

We can now prove the dichotomy.

\begin{proof}[Proof of Theorem~\ref{thm:main}]
If $e_n=1$ for every $n$, then $K_n=E_n\subset K^{\rm ur}$ for every $n$, so
$K_\infty/K$ is unramified.
Otherwise the tower is not unramified.  Proposition~\ref{prop:unbounded-e}
gives $e_n\to\infty$, and Proposition~\ref{prop:unbounded-different} shows that the valuations of the finite-level differents $v_K(\D_{K_n/K})$
are unbounded. The different
criterion recalled at the beginning of this section then implies that
$K_\infty/K$ is deeply ramified.
\end{proof}

\section{The Frobenius-type case}\label{sec:frobenius-type}

We record a useful sharpening of Theorem~\ref{thm:main} when
\[
        f(X)\equiv X^{q}\pmod{\mathfrak m_K},
        \qquad q=p^a,
\]
for some $a\geq1$: then the unramified alternative is trivial.

\begin{lemma}\label{lem:frob-difference}
Assume that
\[
        f(X)\equiv X^q\pmod{\mathfrak m_K},
        \qquad q=p^a.
\]
Let $L/K$ be an unramified algebraic extension, and let $x,y\in \mathcal O_L$ satisfy
$x\equiv y\pmod{\mathfrak m_L}$.  If $x\neq y$, then
\[
        v_K\bigl(f(x)-f(y)\bigr)\geq v_K(x-y)+1.
\]
\end{lemma}

\begin{proof}
Write $f(X)=X^q+\pi_K A(X)$ with $A\in \mathcal O_K[X]$.  Since $q$ is a power of $p$,
\[
        x^q-y^q=(x-y)^q+\pi_K B(x,y)(x-y)
\]
for some $B\in \mathcal O_K[X,Y]$.  Therefore
\[
        f(x)-f(y)=(x-y)^q+
        \pi_K\bigl(B(x,y)(x-y)+A(x)-A(y)\bigr).
\]
Since $A(X)-A(Y)$ is divisible by $X-Y$ in $\mathcal O_K[X,Y]$, the second summand has
valuation at least $1+v_K(x-y)$.  The first summand has valuation
$qv_K(x-y)$.  Because $L/K$ is unramified and $x\equiv y\pmod{\mathfrak m_L}$,
we have $v_K(x-y)\geq1$; hence $qv_K(x-y)\geq v_K(x-y)+1$.  The claimed
inequality follows.
\end{proof}

We now show that this estimate rules out genuinely unramified inverse branches
in the Frobenius-type case.

\begin{proposition}\label{prop:frob-unramified-trivial}
Assume that \(f(X)\equiv X^q\pmod{\mathfrak m_K}\), with \(q=p^a\) and
\(a\geq1\). Let \(t_0\in \mathcal O_K\), and let \((t_n)_{n\geq0}\) be a compatible
backward orbit, \(f(t_{n+1})=t_n\). If \(K_n=K(t_n)\) is unramified over \(K\)
for every \(n\), then in fact
\[
        t_n\in K\qquad\text{for every }n.
\]
Consequently \(K_\infty=K\).
\end{proposition}

\begin{proof}
Since $K_n/K$ is unramified, all $t_n$ lie in $K^{\rm ur}$.  Reducing
$f(t_{n+1})=t_n$ gives
\[
        \overline t_{n+1}^{\,q}=\overline t_n.
\]
The map $x\mapsto x^q$ is an automorphism of every finite residue field and
preserves the subfield $k$.  Since $\overline t_0\in k$, induction gives
$\overline t_n\in k$ for every $n$.

Fix $N\geq0$ and let
$\sigma\in\operatorname{Gal}(K^{\rm ur}/K)$.  We claim that
$\sigma(t_N)=t_N$.  Suppose not, and put
\[
        \delta_m:=\sigma(t_m)-t_m\qquad(m\geq N).
\]
Then $\delta_N\neq0$.  Moreover, if $\delta_m\neq0$, then
$\delta_{m+1}\neq0$, since $\delta_{m+1}=0$ would imply
\[
        \delta_m=\sigma(f(t_{m+1}))-f(t_{m+1})=0.
\]
Thus $\delta_m\neq0$ for all $m\geq N$.

Because the residues $\overline t_m$ lie in $k$, which is fixed by $\sigma$, we
have
\[
        \sigma(t_m)\equiv t_m\pmod{\mathfrak m_{K^{\rm ur}}}.
\]
Applying Lemma~\ref{lem:frob-difference} to
$x=\sigma(t_{m+1})$ and $y=t_{m+1}$ gives
\[
        v_K(\delta_m)
        =v_K\bigl(f(\sigma(t_{m+1}))-f(t_{m+1})\bigr)
        \geq v_K(\delta_{m+1})+1.
\]
Therefore
\[
        v_K(\delta_{N+r})\leq v_K(\delta_N)-r
\]
for all $r\geq0$.  This is impossible for $r$ large, because each nonzero
$\delta_{N+r}$ is integral and hence has nonnegative valuation.  Thus
$\sigma(t_N)=t_N$ for every $\sigma\in\operatorname{Gal}(K^{\rm ur}/K)$, and so
$t_N\in K$.  Since $N$ was arbitrary, all $t_N$ lie in $K$.
\end{proof}

\begin{corollary}\label{cor:frob-trivial-or-deep}
Assume that $f(X)\equiv X^{p^a}\pmod{\mathfrak m_K}$.  Then every compatible
backward orbit starting from $t_0\in \mathcal O_K$ satisfies the sharper alternative
\[
        K_\infty=K
        \qquad\text{or}\qquad
        K_\infty/K\text{ is deeply ramified}.
\]
\end{corollary}

\begin{proof}
This follows from Theorem~\ref{thm:main} and
Proposition~\ref{prop:frob-unramified-trivial}.
\end{proof}

\section{Examples and the APF question}
\label{sec:apf-nonapf}

We finish with examples illustrating the range of the dichotomy. In particular,
the non-unramified alternative may be deeply ramified without being APF at the
algebraic level.

\subsection{A deeply ramified tower which is not APF}

Theorem~\ref{thm:main} proves deep ramification in the non-unramified case.
It is natural to ask whether the conclusion can be strengthened to the APF
property.  In the classical iterated-tower situations studied by Berger and by
Cais--Davis--Lubin~\cite{BergerIterated,CaisDavisLubin}, one starts from a
uniformizer and the relevant polynomials are Eisenstein at every level.  This
gives much more precise ramification control than the argument above, and leads
to APF, often strictly APF, extensions.

In the present generality one cannot replace ``deeply ramified'' by ``APF''.
Fesenko had already shown that deeply ramified extensions need not be APF
\cite[Secs.~2.1--2.4]{FesenkoDeeply}.  The point here is more specific: the
same obstruction appears in an explicit compatible inverse branch of a single
polynomial satisfying \eqref{eq:derivative-hypothesis}.  The obstruction is the
possible coexistence of a deeply ramified wild part with an infinite unramified
residue-theoretic part.  We shall use the standard necessary condition that an
APF extension has finite maximal unramified subextension; see
\cite{WintenbergerNorms,FesenkoDeeply}.

The following example realizes this mechanism in the elementary tower generated by
successive $2p$-th roots of $1+p$: the prime-to-$p$ component produces the
unramified part, while the $p$-part forces ramification.

\begin{proposition}
\label{prop:deep-not-apf}
Assume that $p$ is odd and let $K=\mathbb Q_p$.  Let
\[
        f(X)=X^{2p}.
\]
Then there exists a compatible backward orbit
\[
        f(t_{n+1})=t_n,\qquad t_0=1+p,
\]
such that the associated iterated extension
\[
        K_\infty=\bigcup_{n\geq0}\mathbb Q_p(t_n)
\]
is deeply ramified but not APF.
\end{proposition}

\begin{proof}
Choose a compatible system $(\zeta_n)_{n\geq0}$ of roots of unity such that
\[
        \zeta_0=1,\qquad \zeta_{n+1}^{2p}=\zeta_n,
\]
and such that $\zeta_n$ has exact order $2^n$ for $n\geq1$.  This is possible
because the map $\xi\mapsto \xi^p$ is an automorphism of the group of
$2$-power roots of unity.

Let $w_0=1+p$.  Choose recursively $w_{n+1}$ with residue $1$ and satisfying
\[
        w_{n+1}^{2p}=w_n.
\]
Such a choice is possible in the algebraic closure.  Define
\[
        t_n:=\zeta_n w_n.
\]
Then
\[
        t_{n+1}^{2p}=\zeta_{n+1}^{2p}w_{n+1}^{2p}=\zeta_nw_n=t_n,
\]
so $f(t_{n+1})=t_n$.

We first show that the tower has an infinite unramified subextension.  Since
$w_n\equiv1\pmod{\mathfrak m}$, the residue of $t_n$ is the residue of
$\zeta_n$.  The Teichmuller lift of this residue is $\zeta_n$, because
reduction is injective on roots of unity of order prime to $p$.  Therefore
\[
        \zeta_n\in \mathbb Q_p(t_n).
\]
Thus
\[
        K_\infty\supseteq \bigcup_{n\geq0}\mathbb Q_p(\zeta_n).
\]
The fields $\mathbb Q_p(\zeta_n)$ are unramified, because $2^n$ is prime to
$p$.  Their degrees are unbounded: indeed
\[
        [\mathbb Q_p(\zeta_n):\mathbb Q_p]=\operatorname{ord}_{2^n}(p),
\]
and if these orders were bounded, then some fixed positive power of $p$ would
be congruent to $1$ modulo $2^n$ for infinitely many $n$.  This is impossible.
Hence $K_\infty/K$ contains an infinite unramified subextension.  Therefore
$K_\infty/K$ is not APF.

It remains to show that the tower is not itself unramified.  Since
$\zeta_1=-1\in\mathbb Q_p$, the field $\mathbb Q_p(t_1)$ contains
$w_1=-t_1$.  We claim that $w_1$ is ramified over $\mathbb Q_p$.  Let $a$ be
the square root of $1+p$ congruent to $1$ modulo $p$.  Then
\[
        w_1^p=a.
\]
Moreover
\[
        a\equiv 1+\frac p2\pmod{p^2},
\]
so $a\notin1+p^2\mathbb Z_p$.
If $w_1$ belonged to an unramified extension $E/\mathbb Q_p$, then its residue
would satisfy $\overline{w_1}^{\;p}=1$, hence $\overline{w_1}=1$.  Thus
$w_1\in1+p\mathcal O_E$.  For $p$ odd one has
\[
        (1+p\mathcal O_E)^p\subseteq 1+p^2\mathcal O_E.
\]
It would follow that $a=w_1^p\in1+p^2\mathcal O_E$.  Since
$E/\mathbb Q_p$ is unramified, this contradicts
$a\notin1+p^2\mathbb Z_p$.  Therefore $\mathbb Q_p(w_1)/\mathbb Q_p$ is
ramified, and the iterated tower is not unramified.

By Theorem~\ref{thm:main}, every non-unramified tower of the present type is
deeply ramified.  Combining this with the infinite unramified subextension
proved above shows that $K_\infty/K$ is deeply ramified but not APF.
\end{proof}

The example separates the two mechanisms present in the tower: the \(p\)-part
forces ramification, while the prime-to-\(p\) residue component produces an
infinite unramified subextension.  Thus deep ramification is the optimal
conclusion in this generality.

\subsection{Infinite unramified towers}

The unramified alternative can also be genuinely infinite.

\begin{example}\label{ex:unramified-tower}
Let $K/\mathbb Q_p$ be finite with residue field of cardinality $q$, choose a
prime $\ell\neq p$; the orders of $q$ in
$(\mathbb Z/\ell^n\mathbb Z)^\times$ are unbounded.  Let
\[
        f(X)=X^{\ell p}\in\OK[X]
\]
and take $t_0=1$.  Choose inductively roots of unity $t_n\in\mu_{\ell^n}$, of
exact order $\ell^n$ for $n\geq1$, such
that
\[
        t_{n+1}^{\ell p}=t_n .
\]
This is possible since $\xi\mapsto\xi^\ell$ maps $\mu_{\ell^{n+1}}$ onto
$\mu_{\ell^n}$ and $\xi\mapsto\xi^p$ is an automorphism on $\ell$-power roots of
unity.

Then
\[
        f(t_{n+1})=t_n,
        \qquad K_n=K(t_n)\subset K^{\rm ur}.
\]
Since $\ell\neq p$, all $\ell^n$-th roots of unity lie in unramified
extensions of $K$.  The fields $K(t_n)$ therefore have unbounded residue
degree, and $K_\infty/K$ is an infinite unramified extension.
\end{example}

\subsection{Positive-valuation examples}

\begin{example}
Let $f(X)=X^p$ and take $t_0=\pi_K$.  A compatible choice
\[
        t_{n+1}^p=t_n
\]
produces elements with
\[
        v_K(t_n)=p^{-n}.
\]
Thus the ramification indices are unbounded; the tower is deeply ramified and
its completion is the familiar perfectoid field obtained by adjoining all
$p$-power roots of a uniformizer.
\end{example}

\begin{example}
Let
\[
        f(X)=X^{2p}+X^p+pX\in\mathbb Z_p[X].
\]
Then $f$ is monic and $f'(X)\in p\mathbb Z_p[X]$, while
\[
        \bar f(X)=X^{2p}+X^p
\]
has more than one monomial.  Thus the theorem is not restricted to pure Kummer
towers; if the branch is not unramified, its completion is again perfectoid by
Corollary~\ref{cor:perfectoid-completion}.
\end{example}

\end{document}